\documentclass[11pt]{article}

\usepackage{vmargin,graphicx,amsmath,amssymb,color,subfigure,enumerate,csquotes}
\newtheorem{theorem}{Theorem}[section]
\newtheorem{lemma}[theorem]{Lemma}
\newtheorem{notation}[theorem]{Notation}
\newtheorem{proposition}[theorem]{Proposition}
\newtheorem{corollary}[theorem]{Corollary}

\def\blacksquare{
\thinspace\nobreak \vrule width 5pt height 5pt depth 0pt}
\newtheorem{remark}[theorem]{Remark}

\newenvironment{proof}{\begin{trivlist}
                       \item[]\hspace{0cm}{\bf Proof: }
                       \hspace{0cm} }{\hfill $\blacksquare$
                     \end{trivlist}}

\makeatletter

\@addtoreset{equation}{section}
\makeatother
\def\dr{\partial}
\def\dx{\, {\rm d}}
\def\la{\lambda}
\def\eps{\varepsilon}

\def\R{\mathbb R}

\def\N{\mathbb N}

\def\sL{{\rm L}}
\def\sH{{\rm H}}

\def\A{{\textbf{\textsf{A}}}}

\definecolor{gr}{rgb}   {0.,   0.69,   0.23 }
\definecolor{bl}{rgb}   {0.,   0.5,   1. }
\definecolor{mg}{rgb}   {0.85,  0.,    0.85}
\definecolor{yl}{rgb}   {0.8,  0.7,   0.}

\newcommand{\red}{\mathsf{red}}

\newcommand{\dis}{\mathsf{dis}}
\newcommand{\dist}{\mathsf{dist}}
\newcommand{\Dom}{\mathsf{Dom}}

\newcommand{\model}{\mathsf{model}}

\definecolor{webred}{rgb}{0.75,0,0}
\definecolor{webgreen}{rgb}{0,0.75,0}
\usepackage[citecolor=webgreen,colorlinks=true,linkcolor=webred]{hyperref}

\def\Ca{\mathcal{C}_{\alpha}} 
\def\Pc{\mathcal{P}} 

\def\a{\tfrac\alpha2}

\def\dhmu{\dx\tilde\mu}

\def\LL{{\mathfrak L}_{\alpha,\beta}}
\def\Lla#1{\lambda_{#1}(\alpha,\beta)}
\def\Lpa#1{\psi_{#1}(\alpha,\beta)}
\def\L{{\mathcal L}_{\alpha,\beta}}
\def\Q{\mathcal{Q}_{\alpha,\beta}}
\def\la#1{\tilde\lambda_{#1}(\alpha,\beta)}
\def\pa#1{\tilde\psi_{#1}(\alpha,\beta)}

\def\H{{\mathfrak H}} 
\def\lh#1{\mathfrak l_{#1}} 
\def\fh#1{\mathfrak f_{#1}} 
\def\Opb{\mathcal{P}_{B}}

\def\T{{\ ^{\sf T}}}
\def\span{{\sf{span}}}

\newcommand{\re}{{\rm e}}

\title{Magnetic Neumann Laplacian on a sharp cone}
\author{V. Bonnaillie-No\"el\footnote{IRMAR, ENS Rennes, Univ. Rennes 1, CNRS, UEB, av. Robert Schuman, F-35170 Bruz, France
\texttt{bonnaillie@math.cnrs.fr}} and N. Raymond\footnote{IRMAR, Univ. Rennes 1, CNRS, Campus de Beaulieu, F-35042 Rennes cedex, France
\texttt{nicolas.raymond@univ-rennes1.fr}}}

\begin{document}

\maketitle
\begin{abstract}
This paper is devoted to the spectral analysis of the Laplacian with constant magnetic field on a cone of aperture $\alpha$ and Neumann boundary condition. We analyze the influence of the orientation of the magnetic field. In particular, for any orientation of the magnetic field, we prove the existence of discrete spectrum below the essential spectrum in the limit $\alpha\to 0$ and establish a full asymptotic expansion for the $n$-th eigenvalue and the $n$-th eigenfunction.
\end{abstract}


\section{Introduction}

\subsection{Definition of the main operator}
The right circular cone $\Ca$ of angular opening $\alpha\in\left(0,\pi\right)$ (see Figure~\ref{fig.cone}) is defined  in the cartesian coordinates $(x,y,z)$ by
$$\Ca=\{(x,y,z)\in\R^3,\ z>0,\ x^2+y^2< z^2\tan^2\a\}.$$
\begin{figure}[h!tb]
\begin{center}
\includegraphics[height=6cm]{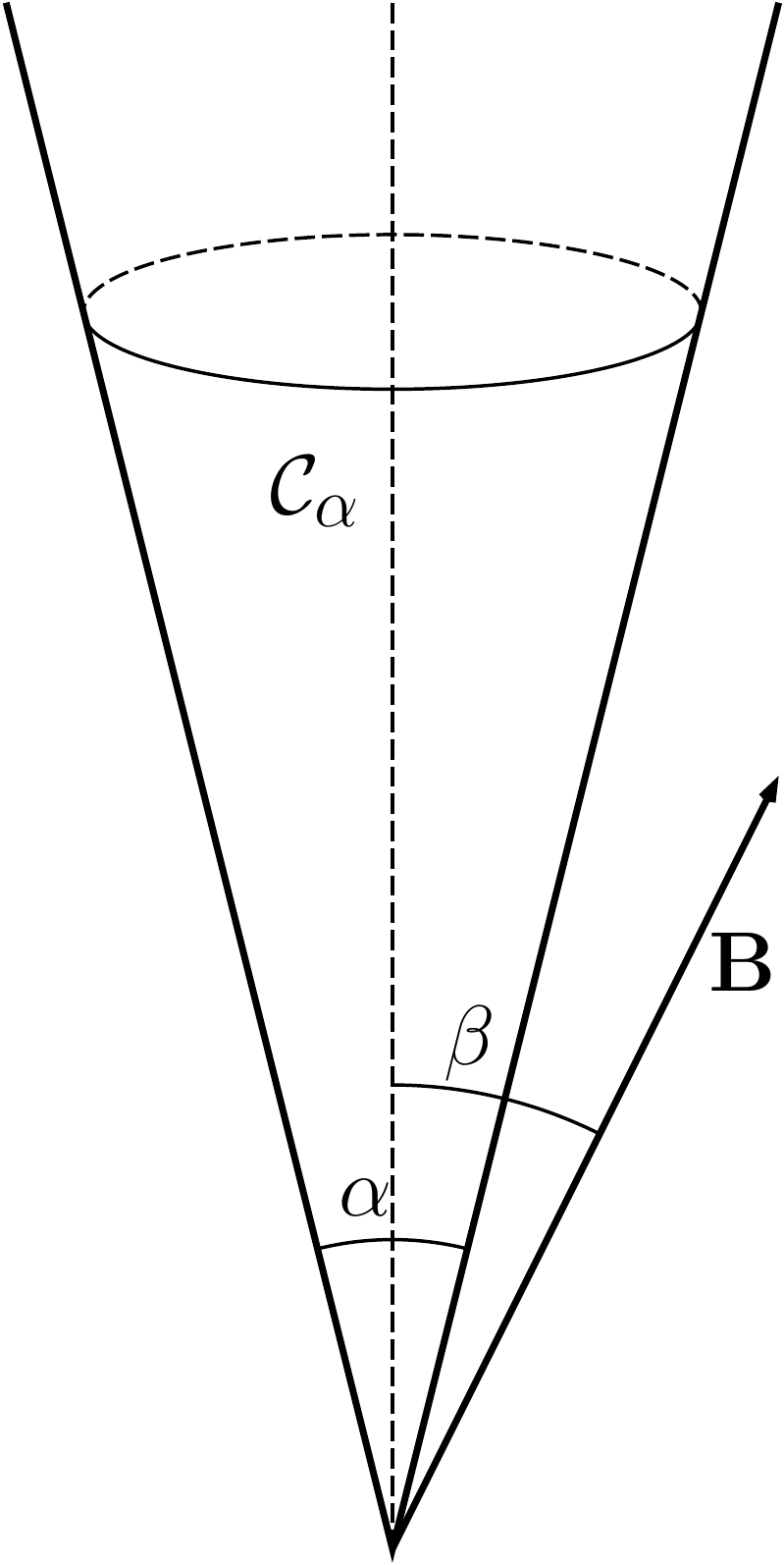}
\caption{Geometric setting.\label{fig.cone}}
\end{center}
\end{figure}
We consider ${\bf B}$ the constant magnetic field which makes an angle $\beta\in\left[0,\frac{\pi}{2}\right]$ with the axis of the cone:
$${\bf B}(x,y,z)=(0,\sin\beta,\cos\beta)^{\sf T},$$
We choose the following magnetic potential $\A$:
$$\A(x,y,z) =\frac{1}{2}{\bf{B}}\times {\bf{x}}=\frac 12(z\sin\beta-y\cos\beta,x\cos\beta,-x\sin\beta)^{\sf T}.$$
We consider $\mathfrak{L}_{\A}=\mathfrak{L}_{\alpha,\beta}$ the Friedrichs extension associated with the quadratic form 
$$\mathfrak{Q}_{\A}(\psi)=\|(-i\nabla+\A)\psi\|^2_{\sL^2(\Ca)},$$
defined for $\psi\in\sH^1_{\A}(\Ca)$ with
$$\sH^1_{\A}(\Ca)=\{u\in\sL^2(\Ca), (-i\nabla+\A)u\in\sL^2(\Ca)\}.$$
The operator $\mathfrak{L}_{\A}$ is $(-i\nabla+\A)^2$ with domain:
$$\sH^2_{\A}(\Ca)=\{u\in\sH^1_{\A}(\Ca), (-i\nabla+\A)^2 u\in\sL^2(\Ca), (-i\nabla+\A)u\cdot\nu=0\, \mbox{ on } \dr\Ca\}.$$
We define the $n$-th eigenvalue $\Lla{n}$ of $\mathfrak{L}_{\A}$ by using Rayleigh quotients: 
\begin{equation}\label{eq.Rayleigh}
\Lla{n}=
   \sup_{\Psi_{1},\ldots,\Psi_{n-1}\in \sH^1_{\A}(\Ca)} \ 
   \inf_{\substack{\Psi\in[\Psi_{1},\ldots,\Psi_{n-1}]^{\bot}\\ \Psi\in \sH^1_{\A}(\Ca),\ \|\Psi\|_{\sL^2(\Ca)}=1 }}  \mathfrak{Q}_{\A}(\Psi) 
=\inf_{\Psi_{1},\ldots,\Psi_{n}\in \sH^1_{\A}(\Ca)}\ \sup_{\substack{\Psi\in[\Psi_{1},\ldots,\Psi_{n}] \\ \|\Psi\|_{\sL^2(\Ca)}=1 }}\mathfrak{Q}_{\A}(\Psi).
\end{equation}
Let $\Lpa{n}$ be a normalized associated eigenvector (if it exists).


\subsection{Expression in spherical coordinates}
The spherical coordinates are naturally adapted to the geometry and we consider the change of variable:
$$\Phi(t,\theta,\varphi):=(x,y,z)=\alpha^{-1/2}(t\cos\theta\sin\alpha\varphi,\ t\sin\theta\sin\alpha\varphi,\ t\cos\alpha\varphi).$$
We denote by $\Pc$ the semi-infinite rectangular parallelepiped
$$\Pc:=\{(t,\theta,\varphi)\in\R^3,\ t>0,\ \theta\in [0,2\pi),\ \varphi\in(0,\tfrac12)\}.$$
Let $\psi\in\sH^1_{\A}(\Ca)$. We write $\psi(\Phi(t,\theta,\varphi))= \alpha^{1/4}\underline\psi(t,\theta,\varphi)$ for any $(t,\theta,\varphi)\in\Pc$ and, using Appendix \ref{A}
and the change of gauge
$$\underline\psi(t,\theta,\varphi)=\exp\left(-i\frac{t^2\varphi}{2}\cos\theta\ \sin\beta\right) \tilde\psi(t,\theta,\varphi),$$
we have
$$\|\psi\|^2_{\sL^2(\Ca)}=
\int_{\Pc}|\tilde\psi(t,\theta,\varphi)|^2\, t^2 \sin\alpha\varphi \dx t \dx\theta \dx\varphi,$$
and:
$$ \mathfrak{Q}_{\A}(\psi)  =\alpha \Q(\tilde\psi),$$
where the quadratic form $\Q$ is defined on the form domain $\sH^1_{\tilde\A}(\Pc)$ by
\begin{equation}\label{eq.Qalpha3D}
\Q(\psi):=\int_{\Pc} \left(|P_{1}\psi|^2+|P_{2}\psi|^2+|P_{3}\psi|^2\right)\dhmu,
\end{equation}
with
\begin{eqnarray*}
P_{1}&=&D_{t}-t\varphi\cos\theta\ \sin\beta,\\
P_{2}&=&{\frac{1}{t\sin\alpha\varphi}\left(D_{\theta}+\frac{t^2\sin^2\alpha\varphi\ \cos\beta}{2\alpha}
+\frac{t^2\varphi\sin\theta\ \sin\beta}2\left(1-\frac{\sin2\alpha\varphi}{2\alpha\varphi}\right)\right)},\\
P_{3}&=&{\frac{1}{\alpha t}D_{\varphi}}.
\end{eqnarray*}
The measure is given by
$$\dhmu=  t^2\sin\alpha\varphi \dx t \dx\theta \dx\varphi,$$
and the form domain by
$$\sH^1_{\tilde\A}(\Pc) =\{\psi\in \sL^2(\Pc,\dhmu), (-i\nabla+\tilde\A)\psi\in\sL^2(\Pc,\dhmu)\}.$$
We consider $\L$ the Friedrichs extension associated with the quadratic form $\Q$:
\begin{eqnarray}\label{Lc}
\L
&=& t^{-2}(D_{t}-t\varphi\cos\theta\sin\beta)t^2(D_{t}-t\varphi\cos\theta\sin\beta)\\
\nonumber &&+\frac{1}{t^2\sin^2(\alpha\varphi)}\left(D_{\theta}+\frac{t^2}{2\alpha}\sin^2(\alpha\varphi)\cos\beta
+\frac{t^2\varphi}{2}\left(1-\frac{\sin(2\alpha\varphi)}{2\alpha\varphi}\right)\sin\beta\sin\theta\right)^2\\
\nonumber &&+\frac{1}{\alpha^2 t^2\sin(\alpha\varphi)}D_{\varphi}\sin(\alpha\varphi) D_{\varphi}.
\end{eqnarray}
We define $\la{n}$ the $n$-th eigenvalue of $\L$ by using the Rayleigh quotients as in \eqref{eq.Rayleigh} and $\pa{n}$ a normalized associated eigenvector if it exists. We have
$$\Lla{n} = \alpha\la{n}.$$

\subsection{Motivation and main result}
This paper is mainly motivated by the theory of superconductivity and the analysis of the Ginzburg-Landau functional. An important result by Giorgi and Philipps (see \cite{GP99}) states that superconductivity disappears when a strong enough exterior magnetic field is applied. This critical intensity above which the superconductor only exists in its \enquote{normal state} is called $H_{C_{3}}$ and is directly related to the lowest eigenvalue of the Neumann realization of the magnetic Laplacian (see \cite{LuPan00a, BonFou07, FouHel10}). In dimension two it has been proved (thanks to semiclassical technics) by Helffer and Morame in \cite{HelMo01} that superconductivity persists longer near the points of the boundary where the curvature is maximal. This fundamental result motivates the investigation of two dimensional domains with corners (see \cite{Jad01, Pan02, Bon05, BD06}). For instance it is proved in \cite{Bon05} that the Neumann Laplacian (with magnetic field of intensity $1$) on the sector with angle $\alpha$ admits a bound state as soon as $\alpha$ is small enough. It is even proved that the first eigenvalues can be approximated by asymptotic series in powers of $\alpha$ the main term being $\alpha/\sqrt 3$ for the first one. In the case of a wedge with aperture $\alpha$ and a magnetic field in the bisector plane of the wedge, Popoff \cite{Popoff} establishes a similar asymptotic expansion for the first eigenvalues and get the same main term for the first eigenvalue (see also \cite{Po13}). In the case of the circular cone $\Ca$ with a magnetic field parallel to the axis ($\beta=0$), it is proved in \cite{BR12} that the lowest eigenvalues always exist as soon as $\alpha$ is small enough and that they admit expansions in the form:
$$\lambda_{n}(\alpha,0)\underset{\alpha\to 0}{\sim}\alpha\sum_{j\geq 0}  \gamma_{j,n} \alpha^{j},\qquad
\mbox{ with }\quad \gamma_{0,n}= \frac{4n-1}{2^{5/2}}.$$
The present paper aims at investigating the influence of the direction of the magnetic field on the spectrum and to answer for instance the following question (in the regime $\alpha\to 0$):
\begin{center}
\enquote{Which is the orientation of the magnetic field which minimizes the first eigenvalues ?}
\end{center}
Before stating our main result concerning the discrete spectrum of $\LL$ let us give a rough estimate (which is sufficient for our purpose) of the infimum of the essential spectrum. Using the Persson's lemma \cite{Persson60}, the bottom of the essential spectrum is given by the behavior at infinity of the operator. In our case, this behavior is described by a Schr\"odinger operator on $\R^3_{+}$ with a constant magnetic. Consequently, with a proof similar to the one of \cite[Proposition 1.2]{BR12}, we have
\begin{proposition}\label{essential}
For all $\alpha\in(0,\pi)$ and $\beta\in\left[0,\frac{\pi}{2}\right]$, we have:
$$s_{\alpha,\beta}:=\inf\sigma_{\rm ess}(\LL)\geq \inf_{\theta\in[0,\frac\pi2]}\sigma(\theta)>0,$$
where $\theta\mapsto\sigma(\theta)$ is the bottom of the spectrum of the Neumann-Schr\"odinger operator on $\R^3_{+}$ with a constant magnetic field that makes an angle $\theta$ with the boundary (see \cite{LuPan00a,BDPR11}).
\end{proposition}
The main result of this paper is the following theorem.
\begin{theorem}\label{main-theo}
Let $\beta\in\left[0,\frac{\pi}{2}\right]$. For all $n\geq 1$, there exist $\alpha_{0}(n)>0$ and a sequence $(\gamma_{j,n})_{j\geq 0}$ such that, for all $\alpha\in(0,\alpha_{0}(n))$, the $n$-th eigenvalue of $\mathfrak{L}_{\alpha,\beta}$ exists and satisfies:
$$\Lla{n}\underset{\alpha\to 0}{\sim}\alpha\sum_{j\geq 0}  \gamma_{j,n} \alpha^{j},\qquad
\mbox{ with }\quad \gamma_{0,n}=\frac{4n-1}{2^{5/2}} \sqrt{1+\sin^2\beta}.$$
\end{theorem}
\begin{remark}
We notice that the main term $\gamma_{0,n}$ in the asymptotic expansion is minimal when $\beta=0$. From the superconductivity point of view this means that superconductivity persists longer when the magnetic field is parallel to the axis of the cone (when $\alpha$ is small enough).
\end{remark}
\begin{remark}
By using the spectral theorem and the quasimodes constructed in Section 2, the corresponding eigenfunctions admit the same kind of expansions in powers of $\alpha$. Contrary to the case analyzed in \cite{BR12}, the eigenfunctions are not axisymmetric when $\beta\neq 0$. Moreover all the powers of $\alpha$ show up in the expansions. 
\end{remark}

\subsection{Strategy of the proof and organization of the paper}
Let us explain the strategy of the proof of Theorem \ref{main-theo}. The first and simplest part of the investigation aims at constructing appropriate quasimodes for $\L$. This can be done by looking for eigenpairs $(\lambda,\psi)$ in the sense of formal power series in $\alpha$ (see Section \ref{formal}). Thanks to the spectral theorem this implies the existence of \textit{some} eigenvalues possessing determined asymptotic expansions (see Proposition \ref{quasimodes}). The main problem is to prove that the formal solutions of the eigenvalue equation are exactly the expansion of the first eigenvalues. In \cite{BR12} we faced the same question, but the analysis was considerably simpler due to the axisymmetry ($\beta=0$). Indeed in the case $\beta=0$ it is possible to prove the axisymmetry of the eigenfunctions (for $\alpha$ small enough) by using a Fourier decomposition with respect to the variable $\theta$ and some rough estimates of Agmon. In fact we will improve these estimates of Agmon in Section \ref{Sec.Agmon} by proving that the length scale on which the eigenfunctions live is $t\sim 1$ or equivalently $z\sim \alpha^{1/2}$. Here the strategy of the Fourier decomposition fails and we shall do something else. If one considers the expression of $\L$ given in \eqref{Lc} we notice (in a heuristic sense) that the second term in penalized by the factor $(t^2\sin^2(\alpha\varphi))^{-1}$. Jointly with our accurate estimates of Agmon, this implies a penalization of $D_{\theta}^2$ which means that the eigenfunctions do not depend on $\theta$ at the main order (see Section \ref{averaging} and especially Lemmas \ref{app-t} and \ref{app-Dt} the proof of which rely on fine commutators computations). Once the asymptotic independence from $\theta$ is established we can replace the first term in \eqref{Lc} by its average with respect to $\theta$ (whereas the term in front of $\sin\theta$ in the second term is obviously small). Therefore the spectral analysis is reduced to an operator which does not depend on $\theta$ anymore (see Section \ref{Sec.comparison} and especially Propositions \ref{comparison} and \ref{gap1}). Finally it remains to apply the analysis of the axisymmetric case of \cite{BR12} (see Proposition \ref{gap2}).

\section{Formal series in $\alpha$}\label{formal}
The aim of the section is to prove the following proposition.
\begin{proposition}\label{quasimodes}
Let $\beta\in\left[0,\frac{\pi}{2}\right]$. For all $n\geq 1$, there exist $\alpha_{0}(n)>0$ and a sequence $(\gamma_{j,n})_{j\geq 0}$ such that, for all $\alpha\in(0,\alpha_{0}(n))$:
$$\dist\left(\alpha\sum_{j=0}^J \gamma_{j,n} \alpha^j, \sigma_{\dis}(\LL) \right)\leq C\alpha^{J+2},\mbox{ with }\quad \gamma_{0,n}=\frac{4n-1}{2^{5/2}}\sqrt{1+\sin^2\beta},$$
where $\sigma_{\dis}(\LL)$ denotes the discrete spectrum of $\LL$.
\end{proposition}
\begin{proof}
We write a formal Taylor expansion in powers of $\alpha$:
$$\L\sim  \sum_{j\geq -2} \alpha^{j} L_{j},$$
where:
$$L_{-2}=t^{-2}(\varphi^{-1}D_{\varphi}\varphi D_{\varphi}+\varphi^{-2}D_{\theta}^2),$$
$$L_{-1}=\cos\beta D_{\theta},$$
$$L_{0}=t^{-2}(D_{t}-t\varphi\cos\theta\sin\beta)t^2(D_{t}-t\varphi\cos\theta\sin\beta)+\frac{t^2\varphi^2\cos^2\beta}{4}+\frac{ \varphi\sin\beta}{3}(\sin\theta D_{\theta}+D_{\theta}\sin\theta).$$
\begin{remark}
We notice that the operator $\Opb=\varphi^{-1}D_{\varphi}\varphi D_{\varphi}+\varphi^{-2}D_{\theta}^2$ defined on the space $\sL^2\left(\left(0,\frac{1}{2}\right)\times[0,2\pi),\varphi\dx \varphi\dx\theta\right)$ with Neumann condition at $\varphi={1}/{2}$ is nothing but the Neumann Laplacian on the disk of center $(0,0)$ and radius ${1}/{2}$.
\end{remark}
We look for quasi-eigenpairs in the form:
$$\lambda\sim\sum_{j\geq -2} \lambda_{j}\alpha^{j},\quad \psi\sim\sum_{j\geq 0} \alpha^j \psi_{j},$$
so that, in the sense of formal series:
$$\L\psi\sim \lambda \psi.$$
\paragraph{Term in $\alpha^{-2}$.}
We have to solve the equation:
$$L_{-2}\psi_{0}=\lambda_{-2}\psi_{0}.$$
We are led to choose $\lambda_{-2}=0$ and $\psi_{0}(t,\theta,\varphi)=f_{0}(t)$.
\paragraph{Term in $\alpha^{-1}$.}
Then, we write:
$$L_{-2}\psi_{1}=(\lambda_{-1}-L_{-1})\psi_{0}=\lambda_{-1}\psi_{0}.$$
For all fixed $t$ the Fredholm alternative gives $\lambda_{-1}=0$ and we choose $\psi_{1}(t,\theta,\varphi)=f_{1}(t)$.
\paragraph{Term in $\alpha^{0}$.}
The crucial equation is:
$$L_{-2}\psi_{2}=(\lambda_{0}-L_{0})\psi_{0}-L_{-1}\psi_{1}=(\lambda_{0}-L_{0})\psi_{0}.$$
For fixed $t$ the Fredholm alternative implies:
$$\langle (\lambda_{0}-L_{0})\psi_{0} ,1\rangle_{\sL^2(\varphi \dx\varphi\dx \theta)}=0.$$
A computation gives:
$$\left(t^{-2}D_{t}^2 t^2 D_{t}^2+2^{-5}(1+\sin^2\beta)t^2\right)f_{0}=\lambda_{0}f_{0}.$$
We can use Corollary \ref{cor.as1term} with $c=2^{-5}(1+\sin^2\beta)$ and we are led to take, for each $n\geq1$, 
$$\lambda_{0}=\frac{4n-1}{2^{5/2}}\sqrt{1+\sin^2\beta},$$
and for $f_{0}$ the corresponding attached eigenfunction.
We take $\psi_{2}$ in the form $\psi_{2}=t^2\tilde\psi_{2}^{\perp}+f_{2}(t)$ where $\tilde\psi_{2}^{\perp}$ is the unique solution of:
$$\Opb\tilde\psi_{2}=(\lambda_{0}-L_{0})\psi_{0}$$
such that $\langle\tilde\psi_{2},1\rangle_{\sL^2(\varphi \dx\varphi\dx \theta)}=0$.
\paragraph{Further terms.}
Step by step, we can determine all the coefficients of the formal series and the conclusion follows from the spectral theorem as we have done in \cite{BR12}.
\end{proof}

\section{Accurate Agmon estimates for $\beta\in\left[0,\frac{\pi}{2}\right]$}\label{Sec.Agmon}

\subsection{Basic estimates}
Before entering into the details of our asymptotic analysis we shall recall basic considerations related to {\it a priori} localization and regularity of the eigenfunctions.
As a classical consequence of the Persson's theorem (see \cite{Persson60}), we can first establish rough Agmon's estimates (see \cite{Agmon82, Agmon85}).
\begin{proposition}\label{aprioriAgmon}
Let $\alpha\in\left(0,\pi\right)$ and $\beta\in\left[0,\frac{\pi}{2}\right]$. There exist $\eps, C>0$ such that for all eigenpair $(\lambda,\psi)$ of $\LL$ satisfying $\lambda<s_{\alpha,\beta}=\inf\sigma_{\rm ess}(\LL)$ we have:
\begin{eqnarray*}
\int_{\Ca} \re^{\eps \sqrt{s_{\alpha,\beta}-\lambda}|{\bf{x}}|} |\psi({\bf{x}})|^2\dx {\bf{x}}\leq C\|\psi\|^2,\\
\Q\left(\re^{\eps \sqrt{s_{\alpha,\beta}-\lambda}|{\bf{x}}|} \psi\right)\leq C\|\psi\|^2.
\end{eqnarray*}
\end{proposition}
\begin{remark}
As we can notice in Proposition \ref{aprioriAgmon} the constants $C$ and $\eps$ {\it a priori} depend on $\alpha$. We will improve these estimates in the next section.
\end{remark}
It is also well-known that the eigenfunctions are in $\sH^2_{loc}(\Ca)$ since $\Ca$ is Lipschitzian and convex (see for instance \cite{Kondratev67}). In fact by using the methods of \cite{D88} (see especially Chapter 6, Section 18 to determine the behavior of the singularities exponents) we can establish the following proposition (by using the elliptic estimates related to the Neumann Laplacian on $\Ca$).
\begin{proposition}
For all $k\geq 3$ there exists $\alpha_{0}>0$ such that for all $\alpha\in(0,\alpha_{0})$, any eigenfuntion belongs to $\sH_{loc}^k(\Ca)$.
\end{proposition}
Then by using the localization estimates of Proposition \ref{aprioriAgmon} and a standard bootstrap argument, we infer:
\begin{proposition}\label{H3}
For all $k\geq 3$ there exists $\alpha_{0}>0$ such that for all $\alpha\in(0,\alpha_{0})$, $\beta\in\left[0,\frac{\pi}{2}\right]$, there exist $\eps>0$ and $C>0$ such that for all eigenpairs $(\lambda,\psi)$ such that $\lambda<s_{\alpha,\beta}$, $\psi$ belongs to $\sH^k(\Ca)$ and:
$$\|\re^{\eps|\bf{x}|}\psi\|_{\sH^k(\Ca)}\leq C\|\psi\|.$$
\end{proposition}
\subsection{Refined estimates of Agmon}
The following propositions provide an improvement of the localization estimates satisfied by the eigenfunctions attached to the low lying eigenvalues: we distinguish between the cases $\beta\in[0,\pi/2)$ and $\beta=\pi/2$.
\begin{proposition}\label{Agmon}
Let $C_{0}>0$. For all $\beta\in\left[0,\frac{\pi}{2}\right)$, there exist $\alpha_{0}>0$, $\eps_{0}>0$ and $C>0$ such that for any $\alpha\in(0,\alpha_{0})$ and for all eigenpair $(\lambda,\psi)$ of $\LL$ satisfying $\lambda\leq C_{0}\alpha$, we have:
\begin{equation}
\int_{\Ca} \re^{2\eps_{0}\alpha^{1/2}|z|}|\psi({\bf x})|^2\dx {\bf{x}}\leq C\|\psi\|^2.
\end{equation}
\end{proposition}
\begin{proof}
Thanks to a change of gauge $\mathfrak{L}_{\A}$ is unitarily equivalent to the Neumann realization of:
$$\mathfrak{L}_{\hat\A}=D_{z}^2+(D_{x}+z\sin\beta)^2+(D_{y}+x\cos\beta)^2.$$
The associated quadratic form is:
$$\mathfrak{Q}_{\hat\A}(\psi)=\int_{\Ca} |D_{z}\psi|^2+|(D_{x}+z\sin\beta)\psi|^2+|(D_{y}+x\cos\beta)\psi|^2\dx x \dx y \dx z.$$
Let us introduce a smooth cut-off function $\chi$ such that $\chi=1$ near $0$ and let us also consider, for $R\geq 1$ and $\eps_{0}>0$:
$$\Phi_{R}(z)=\eps_{0}\alpha^{1/2} \chi\left(R^{-1}z \right) |z|.$$
The Agmon identity gives:
$$\mathfrak{Q}_{\hat\A}(\re^{\Phi_{R}}\psi)=\lambda \|\re^{\Phi_{R}}\psi\|^2+\|\nabla\Phi_{R} \re^{\Phi_{R}}\psi\|^2.$$
There exist $\alpha_{0}>0$ and $\tilde C_{0}>0$ such that for $\alpha\in(0,\alpha_{0})$, $R\geq 1$ and $\eps_{0}\in(0,1)$, we have:
$$\mathfrak{Q}_{\hat\A}(\re^{\Phi_{R}}\psi)\leq \tilde C_{0}\alpha \|\re^{\Phi_{R}}\psi\|^2.$$
We introduce a partition of unity with respect to $z$:
$$\chi_{1}^2(z)+\chi_{2}^2(z)=1,$$
where $\chi_{1}(z)=1$ for $0\leq z\leq 1$ and $\chi_{1}(z)=0$ for $z\geq 2.$ 
For $j=1,2$ and $\gamma>0$, we let:
$$\chi_{j,\gamma}(z)=\chi_{j}(\gamma^{-1}z),$$
so that:
$$\|\chi'_{j,\gamma}\|\leq C\gamma^{-1}.$$
The \enquote{IMS} formula provides:
\begin{equation}\label{IMS-inequality}
\mathfrak{Q}_{\hat\A}(\re^{\Phi_{R}}\chi_{1,\gamma}\psi)+{\mathfrak{Q}_{\hat\A}}(\re^{\Phi_{R}}\chi_{2,\gamma}\psi)-C^2\gamma^{-2}\|\re^{\Phi_{R}}\psi\|^2
    \leq  \tilde C_{0}\alpha \|\re^{\Phi_{R}}\psi\|^2.
\end{equation}
We want to write a lower bound for ${\mathfrak{Q}_{\hat\A}}(\re^{\Phi_{R}}\chi_{2,\gamma}\psi)$. Integrating by slices we have for all $u\in\Dom(\mathfrak{Q}_{\hat\A})$:
\begin{eqnarray}\label{min-Q}
\nonumber\mathfrak{Q}_{\hat\A}(u)
&\geq& \int_{z>0}\left(\int_{\{\sqrt{x^2+y^2}\leq z\tan\frac\alpha2\}}|(D_{x}+z\sin\beta)u|^2+|(D_{y}+x\cos\beta)u|^2\dx x\dx y\right)\dx z\\
&\geq& \int_{z>0}\left(\int_{\{\sqrt{x^2+y^2}\leq z\tan\frac\alpha2\}}|D_{x}\tilde u|^2+|(D_{y}+x\cos\beta)\tilde u|^2\dx x\dx y\right)\dx z\\
\nonumber&\geq&\cos\beta \int_{z>0} \mu\left(z\sqrt{\cos\beta} \tan\frac\alpha2\right)\int_{\{\sqrt{x^2+y^2}\leq z\tan\frac\alpha2\}} |u|^2\dx x\dx y\dx z,
\end{eqnarray}
where we have used the change of gauge (for fixed $z$) $\tilde u=\re^{ixz\sin\beta}u$ and we denote by $\mu(\rho)$ the lowest eigenvalue of the magnetic Neumann Laplacian $D_{x}^2+(D_{y}+x)^2$ on the disk of center $(0,0)$ and radius $\rho$.
By using a basic perturbation theory argument for small $\rho$ (see \cite[Proposition 1.5.2]{FouHel10}) and a semiclassical behaviour for large $\rho$ (see \cite[Section 8.1]{FouHel10}) we infer the existence of $c>0$ such that for all $\rho\geq 0$:
\begin{equation}\label{min-mu}
\mu(\rho)\geq c\min(\rho^2,1).
\end{equation}
We infer: 
$$\mathfrak{Q}_{\hat\A}(\re^{\Phi_{R}}\chi_{2,\gamma}\psi)\geq\int_{z>0} c \cos\beta \min(z^2\alpha^2\cos\beta,1) \int_{\{\sqrt{x^2+y^2}\leq z\tan\frac\alpha2\}}|\re^{\Phi_{R}}\chi_{2,\gamma}\psi|^2\dx x\dx y\dx z.$$
We choose $\gamma=\eps_{0}^{-1}(\cos\beta)^{-1/2}\alpha^{-1/2}$. On the support of $\chi_{2,\gamma}$ we have $z\geq \gamma$. It follows:
$$\mathfrak{Q}_{\hat\A}(\re^{\Phi_{R}}\chi_{2,\gamma}\psi)\geq  c \cos\beta \min(\eps_{0}^{-2}\alpha,1) \|\re^{\Phi_{R}}\chi_{2,\gamma}\psi\|^2.$$
For $\alpha$ such that $\alpha \leq \eps_{0}^2$, we have:
$$\mathfrak{Q}_{\hat\A}(\re^{\Phi_{R}}\chi_{2,\gamma}\psi)\geq  c\alpha\eps_{0}^{-2} \cos\beta \|\re^{\Phi_{R}}\chi_{2,\gamma}\psi\|^2.$$
We deduce that there exist $c>0$, $C>0$ and $\tilde C_{0}>0$ such that for all $\eps_{0}\in(0,1)$ there exists $\alpha_{0}>0$ such that for all  $R\geq 1$ and $\alpha\in(0,\alpha_{0})$:
$$( c \eps_{0}^{-2}\cos\beta-C)\alpha \|\chi_{2,\gamma}\re^{\Phi_{R}}\psi\|^2\leq \tilde C_{0}\alpha \|\chi_{1,\gamma}\re^{\Phi_{R}}\psi\|^2.$$
Since $\cos\beta>0$ and $c>0$, if we choose $\eps_{0}$ small enough, this implies:
$$\|\chi_{2,\gamma}\re^{\Phi_{R}}\psi\|^2\leq \tilde C \|\chi_{1,\gamma}\re^{\Phi_{R}}\psi\|^2\leq \hat C \|\psi\|^2.$$
It remains to take the limit $R\to+\infty$.
\end{proof}

\begin{remark}
Proposition \ref{Agmon} is a refinement of \cite[Proposition 4.1]{BR12} (see also \cite[Remark 4.2]{BR12}) and provides the optimal length scale $z\sim \alpha^{-1/2}$ (or equivalently $t\sim 1$)  for all $\beta\in\left[0,\frac{\pi}{2}\right)$ (in the sense that it exactly corresponds to the rescaling used in the construction of quasimodes).
\end{remark}

In order to analyze the case $\beta=\frac \pi 2$ we will need the following two lemmas the proof of which can be adapted from \cite[Sections 1.5 and 8.1]{FouHel10}. The main point in these lemmas is the uniformity with respect to the geometric constants. The first one is a consequence of perturbation theory.
\begin{lemma}\label{spectrum-ellipse}
Let $0<\delta_{0}<\delta_{1}$. There exist $\rho_{0}>0$ and $c_{0}>0$ such that for $\rho\in(0,\rho_{0})$ and $\delta\in(\delta_{0},\delta_{1})$ we have:
$$\mu_{1}(\delta,\rho)\geq c_{0}\rho^2,$$
where $\mu_{1}(\delta,\rho)$ denotes the first eigenvalue of the Neumann Laplacian with constant magnetic field of intensity $1$ on the ellipse $\mathcal{E}_{\delta,\rho}=\left\{(u,v)\in\R^2 :  u^2+\delta v^2\leq \rho^2\right\}$.
\end{lemma}
The second lemma is a consequence of semiclassical analysis with  semiclassical parameter $h=\rho^{-2}$.
\begin{lemma}\label{spectrum-ellipse2}
Let $0<\delta_{0}<\delta_{1}$. There exist $\rho_{1}>0$ and $c_{1}>0$ such that for $\rho\geq \rho_{1}$ and $\delta\in(\delta_{0},\delta_{1})$ we have:
$$\mu_{1}(\delta,\rho)\geq c_{1}.$$
\end{lemma}

\begin{proposition}\label{pi/2}
Let $C_{0}>0$ and $\beta=\frac{\pi}{2}$. There exist $\alpha_{0}>0$, $\eps_{0}>0$ and $C>0$ such that for any $\alpha\in(0,\alpha_{0})$ and for all eigenpair $(\lambda,\psi)$ of $\LL$ satisfying $\lambda\leq C_{0}\alpha$, we have:
\begin{equation}
\int_{\Ca} \re^{2\eps_{0}\alpha^{1/2}|z|}|\psi({\bf x})|^2\dx {\bf{x}}\leq C\|\psi\|^2.
\end{equation}
\end{proposition}

\begin{proof}
The structure of the proof is the same as for Proposition \ref{Agmon}. The only problem is to replace the inequalities \eqref{min-Q} (since it degenerates when $\beta=\frac{\pi}{2}$) and \eqref{min-mu} (since there is no more reason to consider a magnetic operator on a disk). The main idea to get around the absence of magnetic field in the direction of the cone is to integrate the quadratic form by slices which are not orthogonal to the axis of the cone (see again \eqref{min-Q}): this leads to consider a Laplacian  with a constant (and non trivial) magnetic field on ellipses.

For that purpose, we introduce the following rotation (see Figure \ref{fig.rot}):
\begin{equation}\label{rotation}
x=u,\quad y= \cos\omega\, v-\sin\omega\, w,\quad z=\sin\omega\, v+\cos\omega\, w,
\end{equation}
where $\omega\in\left(0,\frac{\pi}{2}\right)$ is fixed and independent from $\alpha$.
\begin{figure}[h!tb]
\begin{center}
\includegraphics[height=5cm]{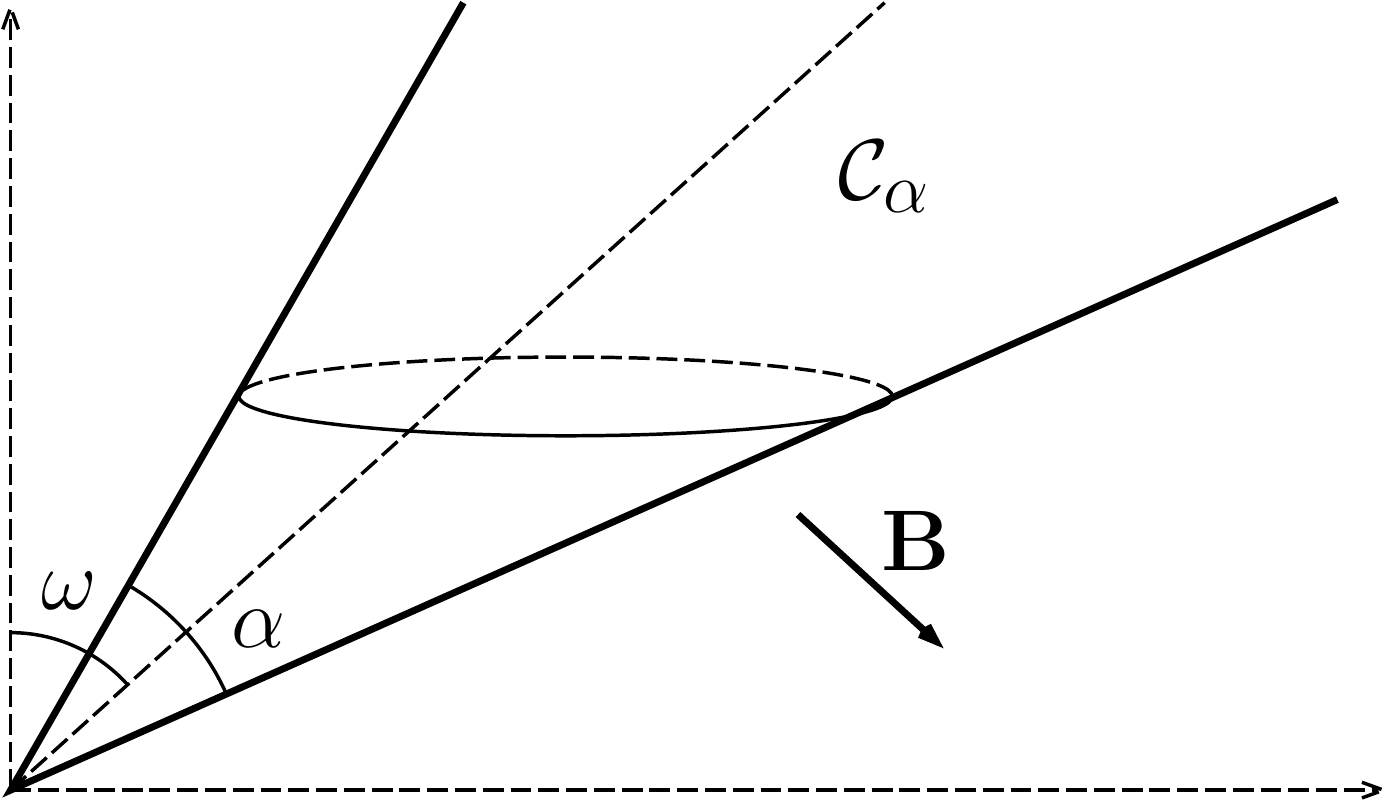}
\caption{Rotation of the cone\label{fig.rot}}
\end{center}
\end{figure}
The $(u,v,w)$-coordinates of ${\bf B}$ are $(0,\cos\omega,-\sin\omega)$. Let us describe the ellipses obtained for fixed $w$.
The cone $\Ca$ is determined by the following inequality:
$$x^2+y^2\leq \tan^2\left(\a\right)\, z^2$$
which becomes:
$$u^2+\delta_{\alpha,\omega}\left(v-\frac{\cos\omega\,\sin\omega\,\left(1+\tan^2\left(\a\right)\right)}{\cos^2\omega-\tan^2\left(\a\right)\sin^2\omega}w\right)^2\leq R_{\alpha,\omega}^2 w^2.$$
where:
$$\delta_{\alpha,\omega}=\cos^2\omega-\tan^2\left(\a\right)\sin^2\omega,$$
\begin{equation}\label{R}
R^2_{\alpha,\omega}=\frac{\tan^2\left(\a\right)}{\cos^2\omega-\tan^2\left(\a\right)\sin^2\omega}.
\end{equation}
Let $\alpha_{0}\in\left(0,\frac{1}{2}\left(\arctan(1/\tan^2\omega\right)\right)$, we notice that 
$$0<\delta_{0}:=\cos^2\omega-\tan^2\left(\tfrac{\alpha_{0}}2\right)\sin^2\omega\leq\delta_{\alpha,\omega}\leq\cos^2\omega=:\delta_{1},\qquad \forall \alpha\in(0,\alpha_{0}).$$ 
With Lemmas \ref{spectrum-ellipse} and \ref{spectrum-ellipse2}, we infer, in the same way as after \eqref{min-Q} and \eqref{min-mu}, the existence of $\alpha_{0}>0$ and $c_{2}>0$ such that for all $\alpha\in(0,\alpha_{0})$ and $\psi\in\Dom\left(\mathfrak{Q}_{\hat{\A}}\right)$\footnote{For a given $w$, we get an ellipse $\mathcal{E}_{\delta_{\alpha,\omega},R_{\alpha,\omega}}$ which is subject to a magnetic field of intensity $\sin\omega$, or equivalently (after dilation) an ellipse $\mathcal{E}_{\delta_{\alpha,\omega},R_{\alpha,\omega}\sqrt{\sin\omega}}$ which is subject to a magnetic field of intensity $1$.}:
$$\mathfrak{Q}_{\hat{\A}}(\psi)\geq\int_{w>0} c_{2}\sin\omega\min(\sin\omega\, R_{\alpha,\omega}^2 w^2,1)\int_{\mathcal{E}_{\delta_{\alpha,\omega},R_{\alpha,\omega}}}|\psi_{\omega}|^2\dx u\dx v     \dx w,$$
where $\psi_{\omega}$ denotes the function $\psi$ after rotation and translation (to get a centered ellipse). 
Then, the proof goes along the same lines as in the proof of Proposition \ref{Agmon} after \eqref{min-mu}. We can express $w$ in terms of the original coordinates $w=z\cos\omega-y\sin\omega$ and $|y|\leq z\tan\left(\a\right)$ so that $z\geq \gamma$ implies:
$$w\geq \gamma\left(\cos\omega-\tan\left(\a\right)\right)\geq\gamma \frac{\cos\omega}{2},$$
as soon as $\alpha$ is small enough. Moreover we deduce from \eqref{R} that  $R^2_{\alpha,\omega}\geq\tan^2\left(\a\right)/\delta_{1}$. 
These considerations are sufficient to conclude as in the proof of Proposition \ref{Agmon}.
\end{proof}
From Propositions \ref{Agmon} and \ref{pi/2}, we can deduce the following corollary.
\begin{corollary}\label{lem.apppsi}
Let $C_{0}>0$ and $\beta\in\left[0,\frac{\pi}{2}\right]$. For all $k\in\N$ there exist $\alpha_{0}>0$, $C>0$ such that for all eigenpairs $(\lambda,\psi)$ of $\L$ such that $\lambda \leq C_{0}$, we have:
$$\|t^{k}\psi\|\leq C\|\psi\|,\quad \Q(t^k\psi)\leq C\|\psi\|^2,\quad \|t^kD_{t}\psi\|\leq C\|\psi\|,\quad \|t^{k}D_{\varphi}\psi\|\leq C\alpha\|\psi\|.$$
\end{corollary}


\section{Commutators and $(\theta,\varphi)$-averaging}\label{averaging}
 This section is devoted to the approximation of the eigenfunctions by their averages with respect to $\theta$.
In order to simplify the analysis let us rewrite $\L$, acting on $\sL^2(\Pc, \dhmu )$ in the following form
\begin{notation}
We will write:
$$\L=\mathcal{L}_{1}+\mathcal{L}_{2}+\mathcal{L}_{3},$$
with
\begin{eqnarray*}
\mathcal{L}_{1}&=&t^{-2}(D_{t}-A_{t})t^2(D_{t}-A_{t}),\\
\mathcal{L}_{2}&=&\frac{1}{t^2\sin^2(\alpha\varphi)}\left(D_{\theta}+A_{\theta,1}+A_{\theta,2}\right)^2,\\
\mathcal{L}_{3}&=&\frac{1}{\alpha^2 t^2\sin(\alpha\varphi)}D_{\varphi}\sin(\alpha\varphi) D_{\varphi},
\end{eqnarray*}
where
\begin{equation}
A_{t}=t\varphi\cos\theta\ \sin\beta,
\end{equation}
\begin{equation}
A_{\theta,1}=\frac{t^2}{2\alpha}\sin^2(\alpha\varphi)\cos\beta,\qquad 
A_{\theta,2}=\frac{t^2\varphi}{2} \left(1-\frac{\sin(2\alpha\varphi)}{2\alpha\varphi}\right)\sin\beta\ \sin\theta.
\end{equation}
We will use the corresponding quadratic forms:
$$\mathcal{Q}_{1}(\psi)=\int_{\Pc} |P_1\psi|^2\dhmu,\qquad 
\mathcal{Q}_{2}(\psi)=\int_{\Pc} |P_2\psi|^2 \dhmu,\qquad 
\mathcal{Q}_{3}(\psi)=\int_{\Pc} |P_3\psi|^2\dhmu,$$
where $P_{1}$, $P_{2}$ and $P_{3}$ are defined by:
$$P_{1}=D_{t}-A_{t},\qquad
 P_{2}=\frac{1}{t\sin(\alpha\varphi)}(D_{\theta}+A_{\theta,1}+A_{\theta,2}),\qquad 
 P_{3}=\frac1{\alpha t} D_{\varphi}.$$
\end{notation}
Let us recall the so-called \enquote{IMS} formula (see \cite{CFKS87}).
\begin{lemma}\label{old-IMS}
Let us  consider a smooth and real function $a$. As soon as each term is well defined, we have:
$$\Re\langle\L\psi, a a \psi \rangle=\Q(a\psi)-\sum_{j=1}^3 \| [a, P_{j}]\psi\|^2.$$
\end{lemma}
We will also need a commutator formula in the spirit of \cite[Section 4.2]{Ray13} (see also \cite{Ray12} where the same commutators method appears).
\begin{lemma}\label{new-IMS}
Let $\psi$ be an eigenfunction for $\L$ associated with the eigenvalue $\lambda$. As soon as each term is well defined, we have the following relation
$$\lambda\|a\psi\|^2=\Q (a\psi)+\sum_{j=1}^3 \langle P_{j}\psi,[P_{j},a^*]a\psi\rangle+\sum_{j=1}^3\langle[a,P_{j}]\psi,P_{j}(a\psi)\rangle,$$
where $a$ is an unbounded operator.
\end{lemma}

\begin{proof}
Formally, we may write:
$$\langle\L\psi, a^*a \psi \rangle= \sum_{j=1}^3 \langle P_{j}\psi, P_{j} a^* a\psi\rangle.$$
Then, we have:
\begin{eqnarray*}
\sum_{j=1}^3 \langle P_{j}\psi, P_{j} a^* a\psi \rangle
&=& \sum_{j=1}^3 \langle P_{j}\psi, a^* P_{j} a\psi\rangle+\langle P_{j}\psi, [P_{j}, a^*] a\psi\rangle\\
&=& \sum_{j=1}^3 \langle a P_{j}\psi, P_{j} a\psi\rangle+\langle P_{j}\psi, [P_{j}, a^*] a\psi\rangle.
\end{eqnarray*}
We infer:
$$\sum_{j=1}^3 \langle P_{j}\psi, P_{j} a^* a\psi\rangle=\Q(a\psi)+\sum_{j=1}^3 \langle [a, P_{j}]\psi, P_{j} a\psi\rangle+\sum_{j=1}^3 \langle P_{j}\psi, [P_{j}, a^*] a\psi\rangle.$$
\end{proof}
\begin{remark}
For instance we can apply Lemma \ref{new-IMS} to $a=tD_{t}$ and $a=t\sin(\alpha\varphi)D_{t}$ thanks to Proposition \ref{H3}.
\end{remark}

\subsection{$\theta$-averaging of $t^k \psi$}

\begin{lemma}\label{app-t}
Let $k\geq 0$ and $C_{0}>0$. There exist $\alpha_{0}>0$ and $C>0$ such that for all $\alpha\in(0,\alpha_{0})$ and all eigenpair $(\lambda,\psi)$ of $\L$ such that $\lambda\leq C_{0}$:
$$\|t^{k} \psi-t^{k}\underline{\psi}_{\theta}\|\leq C\alpha^{1/2}\|\psi\|,$$
with $$\underline{\psi}_{\theta}(t,\varphi)=\frac{1}{2\pi}\int_{0}^{2\pi}\psi(t,\theta,\varphi)\dx\theta.$$
\end{lemma}
\begin{proof}
Let us apply Lemma \ref{old-IMS} with $a=t^{k+1}\sin(\alpha\varphi)$. We get:
$$\Q(t^{k+1}\sin(\alpha\varphi)\psi)=\lambda\|t^{k+1}\sin(\alpha\varphi)\psi\|^2+\|[P_{1},t^{k+1}\sin(\alpha\varphi)]\psi\|^2+\|[P_{3},t^{k+1}\sin(\alpha\varphi)]\psi\|^2.$$
Since $[P_{1},a]=-i(k+1)t^k\sin(\alpha\varphi)$ and $[P_{3},a]=-it^{k}\cos(\alpha\varphi)$, we deduce, using Corollary~\ref{lem.apppsi}, that:
\begin{equation}\label{Q-tksin}
\Q(t^{k+1}\sin(\alpha\varphi)\psi)\leq C \alpha^2\|\psi\|^2+\|t^{k}\cos(\alpha\varphi)\psi\|^2.
\end{equation}
We notice that:
\begin{eqnarray*}
\|P_{3}(t^{k+1}\sin(\alpha\varphi)\psi)\|^2
&=&\frac1{\alpha^{2}}\int_{\Pc} |t^{k} D_{\varphi}(\sin(\alpha\varphi)\psi)|^2 \dhmu\\
&=&\frac1{\alpha^{2}}\int_{\Pc} t^{2k}\left|-i\alpha\cos(\alpha\varphi)\psi+\sin(\alpha\varphi)D_{\varphi}\psi\right|^2\dhmu\\
&\geq& \|t^{k}\cos(\alpha\varphi)\psi\|^2
+\frac2{\alpha^{2}}\Re\left(\int_{\Pc} -i\alpha t^{2k} \cos(\alpha\varphi)\psi \sin(\alpha\varphi)\overline{D_{\varphi}\psi}\dhmu\right).
\end{eqnarray*}
We infer that:
\begin{equation}\label{eq.P3}
\|P_{3}(t^{k+1}\sin(\alpha\varphi))\|^2-\|t^{k}\cos(\alpha\varphi)\psi\|^2\geq -C\|t^{k} \psi\|\|t^kD_{\varphi}\psi\| .
\end{equation}
It follows from \eqref{Q-tksin}, \eqref{eq.P3} and Corollary~\ref{lem.apppsi} that:
$$\mathcal{Q}_{2}(t^{k+1}\sin(\alpha\varphi)\psi)\leq C\alpha\|\psi\|^2.$$
We deduce that : 
$$\int_{\Pc} |(D_{\theta}+A_{\theta,1}+A_{\theta,2})(t^{k}\psi)|^2\dhmu\leq C\alpha\|\psi\|^2,$$
and we get:
$$\frac{1}{2}\int_{\Pc} |D_{\theta}(t^{k}\psi)|^2\dhmu-2\int_{\Pc} |(A_{\theta,1}+A_{\theta,2})t^{k}\psi|^2 \dhmu\leq  C\alpha\|\psi\|^2.$$
Using Corollary \ref{lem.apppsi}, we obtain
$$\int_{\Pc} |(A_{\theta,1}+A_{\theta,2})t^{k}\psi|^2 \dhmu
\leq C \alpha^2\int_{\Pc}t^{2k+4}|\psi|^2\dhmu
\leq C\alpha^2\|\psi\|^2.$$
Therefore we have:
\begin{equation}\label{eq.Dtheta}
\|D_{\theta}(t^{k}\psi)\|^2\leq C\alpha\|\psi\|^2.
\end{equation}
Let us consider $D_{\theta}^2$ on $\sL^2((0,2\pi),\dx\theta)$ (with periodic boundary conditions). The first eigenvalue is simple and equal to 0 and the associated eigenspace is generated by $1$. 
The function  $t^{k} \psi-t^{k}\underline{\psi}_{\theta}$ is orthogonal to $1$. Then, due to the min-max principle, we have
$$\|D_{\theta}(t^{k}\psi)\|^2\geq c_{1}\|t^{k} \psi-t^{k}\underline{\psi}_{\theta}\|^2,$$
with $c_{1}>0$. This last inequality combined with \eqref{eq.Dtheta} completes the proof. 
\end{proof}

\subsection{$\theta$-averaging of $D_{t}\psi$}

\subsubsection{Estimate of $tD_{t}\psi$}

\begin{lemma}\label{bound-on-tDt}
Let $C_{0}>0$. There exist $\alpha_{0}>0$ and $C>0$ such that for all $\alpha\in(0,\alpha_{0})$ and all eigenpair $(\lambda,\psi)$ of $\L$ such that $\lambda\leq C_{0}$:
$$\Q(tD_{t}\psi)\leq C\|\psi\|^2.$$
\end{lemma}
\begin{proof}
Let $a=tD_{t}$. We have $a^*=tD_{t}+\frac{3}{i}$. 
Since $3/i$ commutes with $P_{j}$, we have immediately
$$[P_{j},a^*]=[P_{j},a].$$
Lemma \ref{new-IMS} provides:
$$\Q(tD_{t}\psi)=\lambda\| tD_{t}\psi\|^2-\sum_{j=1}^3\langle P_{j}\psi,[P_{j},tD_{t}]a\psi\rangle-\sum_{j=1}^3\langle [tD_{t},P_{j}]\psi,P_{j}(a\psi)\rangle.$$
Let us compute $[P_{j},tD_{t}]$. We have:
\begin{eqnarray*}
{[P_{1},tD_{t}]} &=&\frac{1}{i}(D_{t}+A_{t})= \frac{1}{i}P_{1}+\frac{2}{i} A_{t}, \\
{[P_{2},tD_{t}]} &=&\frac{1}{it\sin(\alpha\varphi)}(D_{\theta}-A_{\theta,2})=\frac{1}{i}P_{2}-\frac{2}{it\sin(\alpha\varphi)}(A_{\theta,1}+A_{\theta,2}),\\
{[P_{3},tD_{t}]} &=&\frac{1}{i}P_{3}.
\end{eqnarray*}
We infer with Corollary~\ref{lem.apppsi}:
\begin{eqnarray*}
\Q(tD_{t}\psi)&\leq& C\|\psi\|^2+\|P_{1}\psi\|(\|P_{1}a\psi\|+\|2A_{t} a\psi\|)+\|P_{1}(a\psi)\|(\|P_{1}\psi\|+\|2A_{t}\psi\|)\\
&& +\|P_{2}\psi\|\left(\|P_{2}a\psi\|+\left\|\frac{2}{t\sin(\alpha\varphi)}(A_{\theta,1}+A_{\theta,2})a\psi\right\|\right)\\
&& +\|P_{2}a\psi\|\left(\|P_{2}\psi\|+\left\|\frac{2}{t\sin(\alpha\varphi)}(A_{\theta,1}+A_{\theta,2})\psi\right\|\right)+2\|P_{3}\psi\|\|P_{3}a\psi\|.
\end{eqnarray*}
The estimates of Corollary~\ref{lem.apppsi} imply:
$$\Q(tD_{t}\psi)\leq C\|\psi\|^2+C\|\psi\|(\|P_{1}a\psi\|+\|P_{2}a\psi\|+\|P_{3}a\psi\|).$$
It follows that for all $\eps>0$, we have:
$$\Q(tD_{t}\psi)\leq C\|\psi\|^2+\frac C2 \left(\eps^{-1}\|\psi\|^2+\eps\Q(tD_{t}\psi)\right).$$
For $\eps=\frac{1}{C}$, we get:
$$\frac{1}{2}\Q(tD_{t}\psi)\leq \tilde C\|\psi\|^2.$$
\end{proof}

\subsubsection{$\theta$-averaging of $D_{t}\psi$}
This subsection concerns the approximation of $D_{t}\psi$ by its average with respect to $\theta$.
\begin{lemma}\label{app-Dt}
Let $C_{0}>0$. There exist $\alpha_{0}>0$ and $C>0$ such that for all $\alpha\in(0,\alpha_{0})$ and all eigenpair $(\lambda,\psi)$ of $\L$ such that $\lambda\leq C_{0}$, we have:
$$\|D_{t}\psi-D_{t}\underline{\psi}_{\theta}\|\leq C\alpha^{1/2}\|\psi\|.$$
\end{lemma}
\begin{proof}
Taking $a=t\sin(\alpha\varphi)D_{t}$, we have $a^*=t\sin(\alpha\varphi)D_{t}+\frac{3}{i}\sin(\alpha\varphi)=a+\frac 3i\sin(\alpha\varphi)$. 
Applying Lemma~\ref{new-IMS}, we have the relation
\begin{equation}\label{eq.Dtpsi}
\lambda\|a\psi\|^2=\Q(a\psi)+\sum_{j=1}^3\langle P_{j}\psi,[P_{j},a^*](a\psi)\rangle+\sum_{j=1}^3\langle[a,P_{j}]\psi,P_{j}(a\psi)\rangle.
\end{equation}
Now we have to compute the commutators $[P_{j},a]$ and $[P_{j},\sin\alpha\varphi]$. For $j\neq 3$, we have $[P_{j},\sin\alpha\varphi]=0$. Moreover we have: 
\begin{eqnarray*}
{[P_{1},a]=[P_{1},a^*]} &=&\sin(\alpha\varphi)[P_{1},tD_{t}]=\frac{\sin(\alpha\varphi)}{i}\left(P_{1}+2A_{t}\right),\\
{ [P_{2},a]}={[P_{2},a^*]} &=& \sin(\alpha\varphi)[P_{2},tD_{t}]=\frac{\sin(\alpha\varphi)}{i}\left(P_{2}-\frac{2}{t\sin(\alpha\varphi)}(A_{\theta,1}+A_{\theta,2})\right),\\
{[P_{3},a]} &=& -i\sin\alpha\varphi P_{3}-i\cos(\alpha\varphi)D_{t},\\
{[P_{3},a^*]} &=& [P_{3},a]+\frac3i[P_{3},\sin\alpha\varphi]=
-i\sin\alpha\varphi P_{3}-i\cos(\alpha\varphi)D_{t}-3\frac{\cos\alpha\varphi}{t}.
\end{eqnarray*}
The expressions of the commutators, Corollary~\ref{lem.apppsi} and Lemma \ref{bound-on-tDt} imply
\begin{multline}\label{Reste1}
\left| \langle P_{1}\psi,[P_{1},a^*]a\psi\rangle+\langle[a,P_{1}]\psi,P_{1}(a\psi)\rangle\right|\\
\leq C\alpha^2\left(\|P_{1}\psi\|\|(P_{1}+2A_{t})tD_{t}\psi\|+\|(P_{1}+2A_{t})\psi\|\|P_{1} tD_{t}\psi\|\right)\\
\leq C\alpha^2\|\psi\|^2,
\end{multline}
\begin{multline}\label{Reste2}
\left| \langle P_{2}\psi,[P_{2},a^*]a\psi\rangle+\langle[a,P_{2}]\psi, P_{2}(a\psi)\rangle\right|\\
\leq  C\alpha^2\|P_{2}\psi\|\left\| \Big(P_{2}-\frac{2(A_{\theta,1}+A_{\theta,2})}{t\sin(\alpha\varphi)}\Big) tD_{t}\psi  \right\|+ C\alpha^2\|P_{2}tD_{t}\psi\|\left\| \Big(P_{2}-\frac{2(A_{\theta,1}+A_{\theta,2})}{t\sin(\alpha\varphi)}\Big) \psi  \right\|\\
\leq C\alpha\|P_{3}(a\psi)\|^2+C\alpha\|\psi\|^2.
\end{multline}

We have 
$$\left| \langle P_{3}\psi,[P_{3},a^*]a\psi\rangle\right|\leq C\alpha\|P_{3}\psi\|\|P_{3} a\psi\|+C\alpha\|P_{3}\psi\|\|D_{t}(tD_{t}\psi)\|+C\alpha \|P_{3}\psi\|\|D_{t}\psi\|.$$
Corollary~\ref{lem.apppsi} and Lemma \ref{bound-on-tDt} provide
\begin{equation}\label{Reste3}
\left| \langle P_{3}\psi,[P_{3},a^*]a\psi\rangle\right|\leq C\alpha\|\psi\|^2.
\end{equation}
Then, we have a last commutator term to analyze:
$$\langle[a,P_{3}]\psi, P_{3} a\psi\rangle=\langle i\sin(\alpha\varphi)P_{3}\psi, P_{3}a\psi\rangle+\langle i\cos(\alpha\varphi)D_{t}\psi, P_{3} a\psi\rangle.$$
We notice that:
$$\langle i\cos(\alpha\varphi)D_{t}\psi, P_{3} (t\sin\alpha\varphi D_{t})\psi\rangle=\langle i\cos(\alpha\varphi)D_{t}\psi, \sin(\alpha\varphi) P_{3} tD_{t}\psi\rangle-\|\cos\alpha\varphi D_{t}\psi\|^2.$$
We deduce (with Corollary~\ref{lem.apppsi} and Lemma \ref{bound-on-tDt})
\begin{eqnarray}
\left|\langle[a,P_{3}]\psi, P_{3} a\psi)\rangle+\|\cos\alpha\varphi D_{t}\psi\|^2\right|
&\leq& C\alpha\|\psi\|\|P_{3}tD_{t}\psi\|+C\alpha\|D_{t}\psi\|\|P_{3}tD_{t}\psi\|\nonumber\\
&\leq& C\alpha\|\psi\|^2.\label{Reste4}
\end{eqnarray}
Using \eqref{eq.Dtpsi} and the estimates \eqref{Reste1}, \eqref{Reste2}, \eqref{Reste3} and \eqref{Reste4}, we get
$$\sum_{j=1}^3 \mathcal{Q}_{j}(t\sin(\alpha\varphi)D_{t}\psi)-\|\cos(\alpha\varphi)D_{t}\psi\|^{2}\leq C\alpha\|\psi\|^2.$$
Let us estimate
\begin{multline*}
\mathcal{Q}_{3}(t\sin(\alpha\varphi)D_{t}\psi)-\|\cos\alpha\varphi D_{t}\psi\|^2
=\alpha^{-2}\int_{\Pc} |D_{\varphi}\sin(\alpha\varphi) D_{t}\psi|^2 \dhmu-\|\cos\alpha\varphi D_{t}\psi\|^2\\
\geq\frac 2\alpha\int_{\Pc}\Re \left(-i\cos(\alpha\varphi)\sin(\alpha\varphi)D_{t}\psi \overline{D_{\varphi}D_{t}\psi}\right)
\dhmu,
\end{multline*} 
Therefore, we infer, with Corollary~\ref{lem.apppsi}:
$$\mathcal{Q}_{3}(t\sin(\alpha\varphi)D_{t}\psi)-\|\cos\alpha\varphi D_{t}\psi\|^2\geq-C\|D_{t}\psi\|\|D_{\varphi}D_{t}\psi\|\geq -c\alpha\|\psi\|^2.$$
We deduce that:
$$\mathcal{Q}_{2}(t\sin(\alpha\varphi)D_{t}\psi)\leq C\alpha\|\psi\|^2$$
and thus:
$$\int_{\Pc} |(D_{\theta}+A_{\theta,1}+A_{\theta,2})D_{t}\psi|^2 \dhmu\leq C\alpha\|\psi\|^2.$$
The conclusion goes along the same lines as in the proof of Lemma \ref{app-t} (by using also Corollary \ref{lem.apppsi}).
\end{proof}

\section{Reduction to an axisymmetric electro-magnetic Laplacian}\label{Sec.comparison}

\begin{lemma}\label{app-phi}
There exist $\alpha_{0}>0$ and $C>0$ such that for all $\alpha\in(0,\alpha_{0})$ and all eigenpair $(\lambda,\psi)$ of $\L$ such that $\lambda\leq C_{0}$:
$$\|t\psi-t\underline{\psi}_{\varphi}\|\leq C\alpha\|\psi\|.$$
\end{lemma}
\begin{proof}
We have:
$$\Q(t^2\psi)\leq C\|\psi\|^2$$
so that:
$$\mathcal{Q}_{3}(t^2\psi) = \frac{1}{\alpha^2}\|D_{\varphi}(t\psi)\|^2\leq\Q(t^2\psi)\leq C\|\psi\|^2,$$
and thus
$$\|D_{\varphi}(t\psi)\|^2\leq C\alpha^2\|\psi\|^2.$$
We conclude the proof by the min-max principle applied to $t\psi-t{\underline\psi}_{\varphi}$ which is orthogonal to the constant functions.
\end{proof}
Let us introduce an appropriate subspace of dimension $N\geq 1$. Let us consider the family of functions $(\psi_{\alpha,j})_{j=1,\cdots,N}$ such that $\psi_{\alpha,j}$ is a normalized eigenfunction of $\L$ associated with $\lambda_{j}(\alpha,\beta)$ and such that the family is orthogonal for the $\sL^2$ scalar product. We set:
$$\mathfrak{E}_{N}(\alpha)=\underset{j=1,\cdots, N}{\span} \psi_{\alpha,j}.$$
The following proposition reduces the analysis to a model operator which is axisymmetric. 
\begin{proposition}\label{comparison}
There exist $C>0$ and $\alpha_{0}>0$ such that for any $\alpha\in(0,\alpha_{0})$ and all $\psi\in\mathfrak{E}_{N}(\alpha)$, we have 
\begin{equation}
\Q(\psi)\geq(1-\alpha) \mathcal{Q}_{\alpha,\beta}^\model(\psi)-C\alpha^{1/2}\|\psi\|^2,
\end{equation}
where:
\begin{multline*}
\mathcal{Q}_{\alpha,\beta}^\model(\psi)=\\
\int_{\Pc} |D_{t}\psi|^2\dhmu+\frac{1}{2^4}\int_{\Pc} \cos^2(\alpha\varphi) t^2\sin^2\beta |\psi|^2\dhmu+\int_{\Pc} \frac{1}{t^2\sin^2(\alpha\varphi)}|(D_{\theta}+A_{\theta,1})\psi|^2 \dhmu+\|P_{3}\psi\|^2.
\end{multline*}
\end{proposition}
The next two sections are devoted to the proof of Proposition \ref{comparison}.
\subsection{A preliminary reduction}
By definition we can write:
$$\Q(\psi)= \|P_{1}\psi\|^2+\int_{\Pc} \frac{1}{t^2\sin^2(\alpha\varphi)}|(D_{\theta}+A_{\theta,1}+A_{\theta,2})\psi|^2 \dhmu+\|P_{3}\psi\|^2.$$
But we have:
\begin{multline*}
\int_{\Pc} \frac{1}{t^2\sin^2(\alpha\varphi)}|(D_{\theta}+A_{\theta,1}+A_{\theta,2})\psi|^2 \dhmu\\
\geq (1-\alpha)\int_{\Pc} \frac{1}{t^2\sin^2(\alpha\varphi)}|(D_{\theta}+A_{\theta,1})\psi|^2 \dhmu-\alpha^{-1}\int_{\Pc} \frac{1}{t^2\sin^2(\alpha\varphi)}|A_{\theta,2}\psi|^2 \dhmu.
\end{multline*}
Since $|A_{\theta,2}|{(t\sin\alpha\varphi)^{-1}}\leq C\alpha t,$ we infer thanks to Corollary \ref{lem.apppsi}:
\begin{multline*}
\int_{\Pc} \frac{1}{t^2\sin^2(\alpha\varphi)}|(D_{\theta}+A_{\theta,1}+A_{\theta,2})\psi|^2 \dhmu\\
\geq (1-\alpha)\int_{\Pc} \frac{1}{t^2\sin^2(\alpha\varphi)}|(D_{\theta}+A_{\theta,1})\psi|^2 \dhmu-C\alpha\|\psi\|^2.
\end{multline*}
It follows that:
\begin{equation}\label{Qred}
\Q(\psi)\geq(1-\alpha)\mathcal{Q}^\red_{\alpha,\beta}(\psi)-C\alpha\|\psi\|^2,
\end{equation}
where the reduced quadratic form $\mathcal{Q}^\red_{\alpha,\beta}$ is given by:
$$\mathcal{Q}^\red_{\alpha,\beta}(\psi)=\|P_{1}\psi\|^2+\int_{\Pc} \frac{1}{t^2\sin^2(\alpha\varphi)}|(D_{\theta}+A_{\theta,1})\psi|^2 \dhmu+\|P_{3}\psi\|^2.$$

\subsection{Averaging of $P_{1}\psi$}
Let us estimate the difference:
$$\|(D_{t}-A_{t})\psi\|^2-\|(D_{t}-A_{t})\underline{\psi}_{\theta}\|^2.$$
We have:
\begin{multline*}
\left|\|(D_{t}-A_{t})\psi\|^2-\|(D_{t}-A_{t})\underline{\psi}_{\theta}\|^2\right|\\
\leq\int_{\Pc} |D_{t}(\psi-\underline{\psi}_{\theta})-A_{t}(\psi-\underline{\psi}_{\theta})| \left(|(D_{t}-A_{t})\psi|+|(D_{t}-A_{t})\underline{\psi}_{\theta}|\right) \dhmu\\
\leq \|D_{t}(\psi-\underline{\psi}_{\theta})-A_{t}(\psi-\underline{\psi}_{\theta})\| \left(\|(D_{t}-A_{t})\psi\|+\|(D_{t}-A_{t})\underline{\psi}_{\theta}\|\right)\\
\leq \left(\|D_{t}(\psi-\underline{\psi}_{\theta})\|+\|A_{t}(\psi-\underline{\psi}_{\theta})\|\right)\left(C_{0}\|\psi\|+\|D_{t}\underline{\psi}_{\theta}\|+\|A_{t}\underline{\psi}_{\theta}\|\right).
\end{multline*}
We have:
$$\|D_{t}\underline{\psi}_{\theta}\|\leq \|D_{t}\psi\|\leq C\|\psi\|,\quad \|A_{t}\underline{\psi}_{\theta}\|\leq \|t\underline{\psi}_{\theta}\|\leq \|t\psi\|\leq C\|\psi\|.$$
By Lemmas \ref{app-t} and \ref{app-Dt}, we infer:
\begin{equation}\label{approx-(D-A)psi}
\left|\|(D_{t}-A_{t})\psi\|^2-\|(D_{t}-A_{t})\underline{\psi}_{\theta}\|^2\right|\leq C\alpha^{1/2}\|\psi\|^2.
\end{equation}
Then, we can compute:
$$\|(D_{t}-A_{t})\underline{\psi}_{\theta}\|^2=\int_{\Pc} |D_{t}\underline{\psi}_{\theta}|^2\dhmu+\int_{\Pc} t^2\varphi^2\sin^2\beta \cos^2\theta |\underline{\psi}_{\theta}|^2\dhmu.$$
Since $\int_{0}^{2\pi}\cos^2\theta \dx\theta=\frac 12\int_{0}^{2\pi} \dx\theta$, we deduce:
\begin{equation}\label{approx-(D-A)psi-2}
\|(D_{t}-A_{t})\underline{\psi}_{\theta}\|^2=\int_{\Pc} |D_{t}\underline{\psi}_{\theta}|^2\dhmu+\frac{1}{2}\int_{\Pc} t^2\varphi^2\sin^2\beta |\underline{\psi}_{\theta}|^2\dhmu.
\end{equation}
We have:
\begin{equation}\label{approx-(D-A)psi-3}
\left|\|D_{t}\underline{\psi}_{\theta}\|^2-\|D_{t}\psi\|^2\right|\leq C\alpha^{1/2}\|\psi\|^2,\quad \left|\|t\varphi\sin\beta\underline{\psi}_{\theta}\|^2-\|t\varphi\sin\beta\psi\|^2\right|\leq C\alpha^{1/2}\|\psi\|^2.
\end{equation}
We deduce from \eqref{approx-(D-A)psi}, \eqref{approx-(D-A)psi-2} and \eqref{approx-(D-A)psi-3}:
$$\|(D_{t}-A_{t})\psi\|^2\geq \int_{\Pc} |D_{t}\psi|^2\dhmu+\frac{1}{2}\int_{\Pc} t^2\varphi^2\sin^2\beta |\psi|^2\dhmu-C\alpha^{1/2}\|\psi\|^2.$$
By Lemma \ref{app-phi} we have:
$$\left|\|t\varphi\sin\beta\psi\|^2-\|t\varphi\sin\beta\underline{\psi}_{\varphi}\|^2\right|
\leq C \int_{\Pc} |t(\psi-\underline{\psi}_{\varphi})|\left(|t\psi|+|t\underline{\psi}_{\varphi}|\right)\dhmu\leq C\alpha\|\psi\|^2.$$
We infer:
$$\|(D_{t}-A_{t})\psi\|^2\geq \int_{\Pc} |D_{t}\psi|^2\dhmu+\frac{1}{2}\int_{\Pc} t^2\varphi^2\sin^2\beta |\underline{\psi}_{\varphi}|^2\dhmu-C\alpha^{1/2}\|\psi\|^2.$$
We deduce:
$$\|(D_{t}-A_{t})\psi\|^2\geq \int_{\Pc} |D_{t}\psi|^2\dhmu+\frac{c(\alpha)}{2}\int_{\Pc} t^2\sin^2\beta |\underline{\psi}_{\varphi}|^2\dhmu-C\alpha^{1/2}\|\psi\|^2,$$
with
\begin{eqnarray*}
c(\alpha)&=&\frac{\int_{0}^{1/2}\varphi^2\sin\alpha\varphi\dx\varphi}{\int_{0}^{1/2}\sin\alpha\varphi\dx\varphi}=\frac{1}{8}+O(\alpha^2).
\end{eqnarray*}
Finally we deduce:
$$\|(D_{t}-A_{t})\psi\|^2\geq \int_{\Pc} |D_{t}\psi|^2\dhmu+\frac{c(\alpha)}{2}\int_{\Pc} t^2\sin^2\beta |\psi|^2\dhmu-C\alpha^{1/2}\|\psi\|^2.$$
With \eqref{Qred}, we infer:
$$\mathcal{Q}_{\alpha,\beta}(\psi)\geq (1-\alpha)\mathcal{Q}^\model_{\alpha,\beta}(\psi)-C\alpha^{1/2}\|\psi\|^2,$$
where we have used that:
$$\int_{\Pc} t^2\sin^2\beta |\psi|^2\dhmu=\int_{\Pc} t^2\cos^2(\alpha\varphi)\sin^2\beta |\psi|^2\dhmu+O(\alpha^2)\|\psi\|^2.$$
This concludes the proof of Proposition \ref{comparison}.

\subsection{Proof of Theorem \ref{main-theo}}
From Proposition \ref{comparison} and from the min-max principle we deduce:
\begin{proposition}\label{gap1}
Let $N\geq 1$. There exist $\alpha_{0}>0$ such that for all $\alpha\in(0,\alpha_{0})$:
\begin{equation}\label{spectral-comparison}
\tilde\lambda_{N}(\alpha,\beta)\geq (1-\alpha)\lambda^\model_{N}(\alpha,\beta)-C\alpha^{1/2},
\end{equation}
where $\lambda_{N}^\model(\alpha,\beta)$ is the $N$-th eigenvalue of the Friedrichs extension on $\sL^2(\Pc,\dhmu)$ associated with $\mathcal{Q}^\model_{\alpha,\beta}$ which is denoted by $\L^{\model}$:
\begin{multline*}
\L^{\model}=t^{-2}D_{t}t^2D_{t}+\frac{\sin^2 \beta\cos^2(\alpha\varphi)}{2^4}t^2\\
+\frac{1}{t^2\sin^2(\alpha\varphi)}\left(D_{\theta}+\frac{t^2}{2\alpha}\sin^2(\alpha\varphi)\cos\beta\right)^2+\frac{1}{\alpha^2 t^2\sin(\alpha\varphi)}D_{\varphi}\sin(\alpha\varphi) D_{\varphi}.
\end{multline*}
\end{proposition}
The operator $\alpha\L^{\model}$ is the expression in the coordinates $(t,\theta,\varphi)$ of the Neumann electro-magnetic Laplacian on $\sL^2(\Ca)$ with magnetic field $(0,0,\cos\beta)$ and electric potential $V_{\alpha,\beta}({\bf{x}})=2^{-4}\alpha^2\sin^2\beta |z|^2$. This operator on $\sL^2(\Ca)$ reads:
$$\left(D_{x}-\frac{y\cos\beta}{2}\right)^2+\left(D_{y}+\frac{x\cos\beta}{2}\right)^2+D_{z}^2+2^{-4}\alpha^2\sin^2\beta |z|^2.$$
We notice that the magnetic field and the electric potential are axisymmetric so that we are reduced to exactly the same analysis as in \cite{BR12}. In particular we can prove that the eigenfunctions of $\L^{\model}$ associated to the first eigenvalues do not depend on $\theta$ as soon as $\alpha$ is small enough and satisfy estimates of the same kind as in Corollary \ref{lem.apppsi}.  Therefore the spectral analysis of $\L^\model$ is reduced to the one of
\begin{equation*}
t^{-2}D_{t}t^2D_{t}+\frac{\sin^2 \beta\cos^2(\alpha\varphi)}{2^4}t^2+\frac{t^2}{4\alpha^2}\sin^2(\alpha\varphi)\cos^2\beta+\frac{1}{\alpha^2 t^2\sin(\alpha\varphi)}D_{\varphi}\sin(\alpha\varphi) D_{\varphi}.
\end{equation*}
After an averaging argument with respect to $\varphi$ we infer the following proposition:
\begin{proposition}\label{gap2}
Let $n\geq 1$. There exist $\alpha_{0}>0$ such that for all $\alpha\in(0,\alpha_{0})$:
$$\lambda_{n}^\model(\alpha,\beta)=\frac{4n-1}{2^{5/2}}\sqrt{1+\sin^2\beta}+O(\alpha^{1/2}).$$
\end{proposition}
Jointly with \eqref{spectral-comparison} and Proposition \ref{quasimodes} this proves Theorem \ref{main-theo}.

\appendix
\section{Spherical magnetic coordinates}\label{A}
In dilated spherical coordinates $(t,\theta,\varphi)\in\Pc$ such that
$$(x,y,z)=\Phi(t,\theta,\varphi)=\alpha^{-1/2}(t\cos\theta\sin\alpha\varphi,\ t\sin\theta\sin\alpha\varphi,\ t\cos\alpha\varphi),$$
the magnetic potential reads
$$\A(t,\theta,\varphi)=\frac{\alpha^{-1/2}t}{2}(
\cos\alpha\varphi\ \sin\beta-\sin\theta\ \sin\alpha\varphi\ \cos\beta, \cos\theta\ \sin\alpha\varphi\ \cos\beta, -\cos\theta\ \sin\alpha\varphi\ \sin\beta)^{\sf T}.$$
The Jacobian matrix associated with $\Phi$ is
$$D\Phi(t,\theta,\varphi)=\alpha^{-1/2}\begin{pmatrix}
\cos\theta\sin\alpha\varphi&-t\sin\theta\sin\alpha\varphi&\alpha\ t\cos\theta\cos\alpha\varphi\\
\sin\theta\sin\alpha\varphi&t\cos\theta\sin\alpha\varphi&\alpha\ t\sin\theta\cos\alpha\varphi\\
\cos\alpha\varphi&0&-\alpha\ t\sin\alpha\varphi
\end{pmatrix}.$$
We can compute
$$(D\Phi)^{-1}(t,\theta,\varphi)=\alpha^{1/2}t^{-1}\begin{pmatrix}
t\cos\theta\sin\alpha\varphi&t\sin\theta\sin\alpha\varphi&t\cos\alpha\varphi\\
-\sin\theta(\sin\alpha\varphi)^{-1}&\cos\theta(\sin\alpha\varphi)^{-1}&0\\
\frac 1\alpha\cos\theta\cos\alpha\varphi&\frac 1\alpha\sin\theta\cos\alpha\varphi&-\frac 1\alpha\sin\alpha\varphi
\end{pmatrix}.$$
Consequently, the metric becomes
$$G=(D\Phi)^{-1}\,\T(D\Phi)^{-1}=\alpha
\begin{pmatrix}
1&0&0\\
0&t^{-2}(\sin\alpha\varphi)^{-2}&0\\
0&0&(\alpha t)^{-2}
\end{pmatrix}.
$$
The change of variables leads to define the new magnetic potential
\begin{eqnarray}
\tilde \A(t,\theta,\varphi)&=&\T D\Phi\ \A(t,\theta,\varphi)\nonumber\\
&=&\alpha^{-1}\frac{t^2}2\left(0,\sin^2\alpha\varphi\ \cos\beta-\cos\alpha\varphi\ \sin\alpha\varphi\ \sin\theta\ \sin\beta,\cos\theta\ \sin\beta\right)\\
&=&\alpha^{-1}\frac{t^2}2\left(0,\sin^2\alpha\varphi\ \cos\beta-\frac{1}2\sin2\alpha\varphi\ \sin\theta\ \sin\beta,\cos\theta\ \sin\beta\right).\label{tildeA}
\end{eqnarray}
Let $\psi$ be a function in the form domain $\sH^1_{\A}(\Ca)$ of the Schr\"odinger operator $(-i\nabla+\A)^2$ and $\tilde\psi(t,\theta,\varphi) = \alpha^{-1/4}\psi(x,y,z)$ (where $\alpha^{-1/4}$ is a normalization coefficient). The change of variables on the norm and quadratic form reads
$$\|\psi\|^2_{\sL^2(\Ca)}=\int_{\Pc}|\tilde\psi(t,\theta,\varphi)|^2\, t^2\sin\alpha\varphi \dx t \dx\theta \dx\varphi,$$
\begin{multline*}
\int_{\Ca}\left|(-i\nabla+\A)\psi(x,y,z)\right|^2\dx x\dx y\dx z\\
=\int_{\Pc} \langle G(-i\nabla_{t,\theta,\varphi}+\tilde\A)\tilde\psi,
    (-i\nabla_{t,\theta,\varphi}+\tilde \A)\tilde\psi\rangle\ t^2\sin\alpha\varphi \dx t \dx\theta \dx\varphi\\
= \alpha\int_{\Pc} \Bigg(|\dr_{t}\tilde\psi|^2
+\frac{1}{t^2\sin^2\alpha\varphi}\left|\left(-i\dr_{\theta}+\frac{t^2\sin^2\alpha\varphi\ \cos\beta}{2\alpha}-\frac{t^2\sin2\alpha\varphi\ \sin\theta\ \sin\beta}{4\alpha}\right)\tilde\psi\right|^2\\
+\frac{1}{\alpha^2t^2}\left|\left(-i\dr_{\varphi}+\frac{t^2}{2}\cos\theta\ \sin\beta\right)\tilde\psi\right|^2\Bigg)
\, t^2\sin\alpha\varphi \dx t \dx\theta \dx\varphi.
\end{multline*}


\section{Model operators}\label{B}
\begin{proposition}\label{harmonic-radial}
Let $\H_{\omega}$ be defined on $\sL^2(\R_{+},t^2\dx t)$ by
$$\H_{\omega} = -\frac{1}{t^{2}}\partial_{t}t^2\partial_{t}+t^2+\frac{\omega^2}{t^2}.$$
The eigenpairs of $\H_{\omega}$ are $(\lh{n}^\omega,\fh{n}^\omega)_{n\geq1}$ given by
$$\lh{n}^\omega=4n-2+\sqrt{1+4\omega^2},\qquad \fh{n}(t)=P^\omega_{n}(t^2)\ \re^{-t^2/2},$$
with $P^\omega_{n}$ a polynomial function of degree $n-1$.
\end{proposition}

\begin{corollary}\label{cor.as1term}
For $c>0$ the eigenpairs of the operator 
$$\tilde\H = -\frac{1}{t^{2}}\partial_{t}t^2\partial_{t}+c t^2,$$
defined on $\sL^2(\R_{+},t^2\dx t)$ are given by
$$\lh{n}=c^{1/2}(4n-1),\qquad \fh{n}(t)=c^{1/4}\fh{n}^0(c^{1/4}t)=c^{1/4}P_{n}^0(c^{1/4}t)\ \re^{-c^{1/2} t^2/2}.$$
\end{corollary}

\paragraph{Acknowledgments} This work was partially supported by the ANR (Agence Nationale de la Recherche), project {\sc Nosevol} n$^{\rm o}$ ANR-11-BS01-0019.

\def\cprime{$'$}
\bibliographystyle{mnachrn}

\end{document}